\documentclass[reqno,final]{amsart}
\usepackage{natbib}  %Nature-like bibliography
\usepackage{fancyhdr} %Headers
\usepackage{color} %Color definition
\usepackage{hyperref} %Internal and external links
\usepackage{graphicx} %Graphics inclusion

\definecolor{aleacolor}{rgb}{0.16,0.59,0.78}

\hypersetup{
breaklinks,
colorlinks=true,
linkcolor=aleacolor,
urlcolor=aleacolor,
citecolor=aleacolor}

\renewcommand{\headheight}{11pt}
\renewcommand{\cite}{\citet}

\theoremstyle{plain}
\newtheorem*{thm*}{Theorem} %add by author
\newtheorem{theorem}{Theorem}[section]                                          
                          
\newtheorem{lemma}[theorem]{Lemma}
\newtheorem{corollary}[theorem]{Corollary}

\theoremstyle{definition}

\theoremstyle{remark}
\newtheorem{remark}[theorem]{Remark}

\makeatletter \@addtoreset{equation}{section} \makeatother
\renewcommand\theequation{\thesection.\arabic{equation}}
\renewcommand\thefigure{\thesection.\arabic{figure}}
\renewcommand\thetable{\thesection.\arabic{table}}

\newcommand{\aleaIndex}[1]{\href{http://alea.impa.br/english/index_v#1.htm}{\bf #1}}
\eheader{Alea}{\aleaIndex{7}}{2010}{137}{149}

\elogo{\framebox[5.8cm][c]{\footnotesize \parbox[c]{5.5cm}{
The original article is published by the \href{http://alea.impa.br/english/index_v7.htm}{Latin American Journal of Probability and Mathematical Statistics}
} } \parbox[c]{3cm}{\includegraphics[width=3cm]{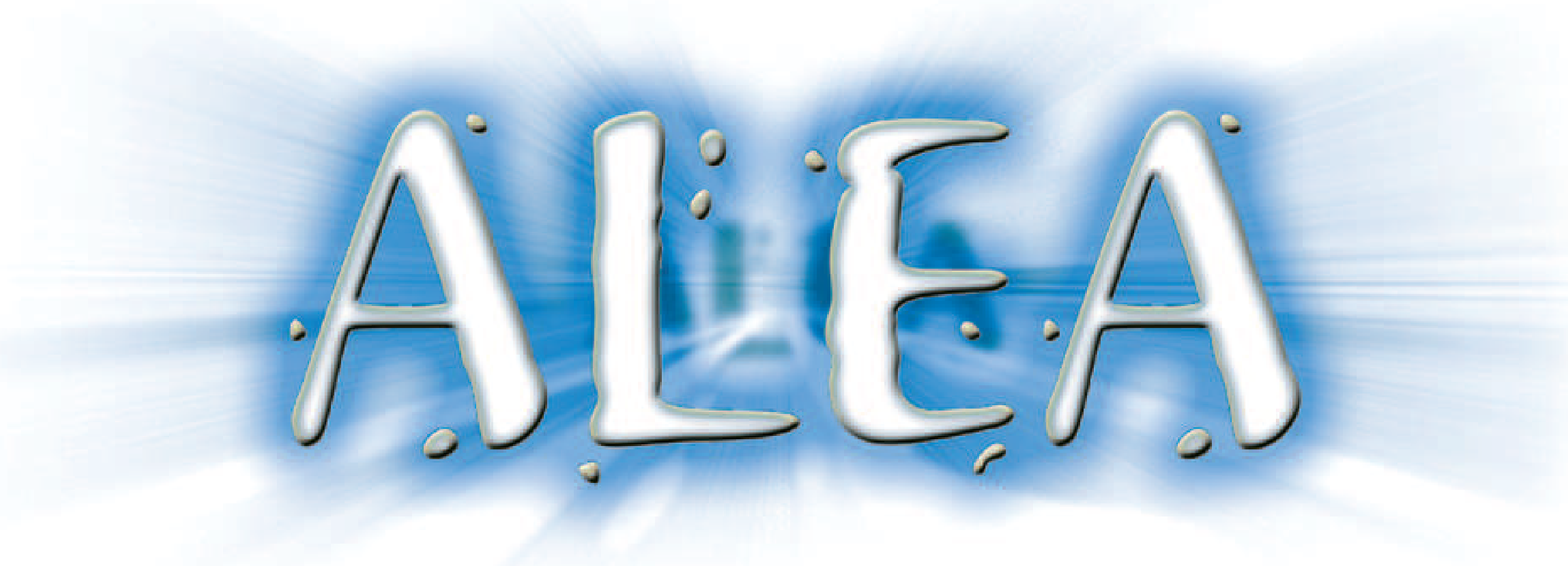}}}

\usepackage{amsthm}
\usepackage{amssymb}

\makeatother

\begin{document}

\title{Nonfixation for Activated Random Walks}
\begin{abstract}
We consider the activated random walk (ARW) model where particles
follow the path of a general Markov process on a general graph. We
prove that ARW, regardless of sleep rate, dominates a simpler process,
multiple source internal aggregation (MSIA), and use this to formulate
a deterministic sufficient condition on initial occupations for nonfixation
of ARW and similar variants. In particular, on bounded degree graphs,
initial occupation density greater than one almost surely implies
nonfixation, where independence requirements are weakened to ergodic
in the case of Euclidean lattices. Finally, we prove the critical
density for the infinite sleep rate ARW is positive for all dimensions. 
\end{abstract}

\author{Eric Shellef}

\address{The Weizmann Institute of Science\\
76100 Rehovot\\
Israel}

\email{shellef@gmail.com}

\keywords{Activated Random Walk, nonfixation}

\subjclass[2000]{60K35}

\date{November 4, 2009; accepted May 13, 2010}

\maketitle

\section{Introduction}

Given a graph, the activated random walks (ARW) process is an interacting
system of \emph{active} and \emph{sleeping} particles on the vertices
of the graph. Active particles perform independent, rate one, random
walks, while sleeping particles stay put. A sleeping particle remains
sleeping as long as it is alone at a vertex. Active particles fall
asleep independently at a rate $\lambda>0$. However, if two or more
particles occupy the same vertex, they will all be active. Thus a
sleeping particle at a vertex $v$ will be activated by an active
particle that jumps to $v$, and an active particle will never fall
asleep if it is not alone at a vertex. All particles begin as active
from some initial occupation state at time $t=0$. We call a vertex
\emph{fixating }if\emph{ }for some finite time onward, no active particle
visits the vertex. A vertex is \emph{nonfixating} if no such time
exists. A graph is nonfixating for an initial occupation if all vertices
are nonfixating almost surely. Otherwise, it is fixating for that
initial occupation.

The main theme explored is the connection between initial particle
placement and nonfixation. In particular, if the initial occupation
of each vertex is an iid r.v. $X$ with mean $\mu$, is there a phase
transition from fixation to nonfixation as we increase $\mu$? Several
papers treating this question used $\mathbb{Z}^{d}$ as a setting
and took $X$ to have a Poisson distribution with mean $\mu$. An
initial result in this framework can be found in \cite{kesten2006phase}
where a model similar to ARW, in which sleeping particles are not
static, was studied in detail. In Remark 5 of the introduction of
\cite{kesten2006phase}, ARW itself with {}``sleepier'' initial
conditions is treated, and a general upper bound proven on the critical
density for global fixation. Specifically, it is shown that for some
large initial density $\mu(d)$, there is a positive probability for
a system starting with all particles sleeping except for a single
active one to never reach a state where all particles are sleeping,
regardless of how high the sleep rate is. In \cite{rolla2009absorbing},
ARW is studied by Rolla and Sidoravicius, again on $\mathbb{Z}^{d}$
with $X$ Poisson. A monotonicity theorem proved in the same paper,
implies that there is at most one phase transition in the density
$\mu$ for the ARW process on a graph. An explicit upper bound and
a non-trivial lower bound on the critical phase transition density
in one dimension are given. In \cite{amir2009fixation}, Amir and
Gurel-Gurevich show that on transitive graphs satisfying the unimodularity
condition, if each vertex is fixating, then eventually all particles
fall asleep. They proceed to show that a.s. there is no fixation on
such graphs for initial occupation density larger than one (compare
with Theorem \ref{thm:uncor_mean_gt_1}).

In this paper we consider ARW in a general form, allowing the particles
to follow the path of a general Markov process on the graph. Using
a dominating coupling of a simpler finite process (Lemma \ref{lem:coupling_with_MSIA}),
we formulate in Theorem \ref{thm:Delta_deviation} a deterministic
sufficient condition on initial occupations for nonfixation, regardless
of sleep rate $\lambda$. We use this condition to prove nonfixation
in several scenarios. In Theorem \ref{thm:uncor_mean_gt_1}, we prove
that on all bounded degree graphs, an initial occupation density greater
than one almost surely implies nonfixation, answering a question posed
by Rolla and Sidoravicius (\citealp{rolla2008generalized,rolla2009absorbing}).
The iid assumption on the initial occupation of each vertex can be
significantly weakened, and in the special case of $\mathbb{Z}^{d}$,
we show that ergodic initial occupation suffices (Theorem \ref{thm:ergodic}).
In the next section, we prove the critical density for the infinite
sleep rate ARW is positive for all dimensions. In the appendix we
prove a general lemma for simple random walks on a network that may
be of independent interest. The lemma gives a lower bound for the
average exit time from a set starting at some vertex using the average
number of visits to that vertex before exit.

Note that the nonfixating condition (and thus the results of section
\ref{sec:Nonfixation-for-Random}) holds for a large class of ARW
variants, since the only assumption used is that if a vertex is occupied
by more than one particle, there must be a time when one of the particles
hops from that vertex.

\section{Notation and Nonfixation Criterion}

Let $S$ be a countable graph, let $\eta_{0}:S\to\mathbb{N}\cup\left\{ 0\right\} $
be an initial occupation for a ARW process on $S$, and let $P(\cdot,\cdot)$
be a transition kernel for a Markov process on $S$. Each particle
is assigned an independent rate one Poisson clock, which together
with $P$, determines the evolution of the particle in the ARW when
it is active. We denote the triplet $\left(S,\eta_{0},P\right)$ by
$\eta$, and write $\mathbb{P}_{(\lambda)}^{\eta}\left[\cdot\right]$
for the law of the ARW process with kernel $P$ on $S$ with initial
occupation $\eta_{0}$ and $\lambda$ a positive or infinite sleep
rate. We write $\lambda=\infty$ for the process where a particle
is sleeping if and only if it is alone at a vertex. 

The question of existence of the process, and ruling out the possibility
that an infinite number of particles reach some vertex in a finite
time is not treated here. See \cite{andjel1982invariant} for conditions
and a construction which can easily be adapted to the ARW process.

Fix a vertex $x\in S$. For an integer $r$, $\mathcal{A}_{r}=\mathcal{A}_{r}(x)$
is the event that $x$ is visited by an active particle at least $r$
distinct times.

We call $\eta$ \emph{nonfixating} at $x$ if $\mathbb{P}_{(\lambda)}^{\eta}\left[\mathcal{A}_{r}(x)\right]=1$
for every $r\in\mathbb{N}$ and every positive or infinite $\lambda$.
We say $\eta$ is nonfixating if it is nonfixating at all its vertices.
From here on, we omit $\lambda$ from the notation as the results
hold regardless of $\lambda$.

Given a graph $S$ and kernel $P$, we define the following multiple
source internal aggregation (MSIA) process. We begin with a finite
number of indistinguishable {}``explorers'' occupying different
vertices of $S$ according to an initial occupation $\mbox{\ensuremath{\gamma_{0}:S\to\mathbb{N}\cup\left\{ 0\right\} }}$.
If a vertex contains more than one explorer, then one of the explorers
at that vertex begins a walk according to kernel $P$ until reaching
an unoccupied vertex, where he remains forever. Assume some fixed
ordering of the vertices of $S$. The explorers perform these walks
one at a time by this order. This continues until each vertex has
at most one explorer. The law of the random subset of occupied vertices
at this stage is not dependent on the order in which explorers are
chosen, as was shown in \cite{diaconis1991growth}, where this growth
model is described. Moreover, the below domination statement is true
for any fixed order. Thus when we discuss an MSIA process we implicitly
assume some arbitrary ordering of the vertices.

We identify the triplet $\left(S,\gamma_{0},P\right)$ by $\gamma$
and write $\mathcal{P}^{\gamma}\left[\cdot\right]$ for the law of
the MSIA process on $S$ with kernel $P$ and initial occupation $\gamma_{0}$.
Assume we have two triplets $\eta=\left(S,\eta_{0},P\right)$ and
$\gamma=\left(S,\gamma_{0},P\right)$ such that $\sum_{x\in S}\gamma_{0}(x)$
is finite and $\gamma_{0}\le\eta_{0}$, i.e. there are initially no
less particles than explorers for each vertex in $S$. In Lemma \ref{lem:coupling_with_MSIA},
which uses ideas found in \cite{lawler1992internal,diaconis1991growth,rolla2009absorbing},
we show that the number of visits by an explorer to a vertex $x$
in $\gamma$-MSIA is stochastically dominated by the number of visits
of particles to $x$ in $\eta$-ARW.

The advantage of this coupling technique is that it assumes only a
very basic property of the ARW and is thus valid for many ARW variations,
see Remark \ref{rem:Sole_ARW_assumption}.

In the context of the MSIA process, we write $\mathcal{A}_{r}(x)$
for the event that vertex $x$ is visited by at least $r$ explorers.
\begin{lemma}
\label{lem:coupling_with_MSIA}Let $\eta=\left(S,\eta_{0},P\right)$
define an ARW, let $\gamma=\left(S,\gamma_{0},P\right)$ with\break ${\sum_{x\in S}\gamma_{0}(x)<\infty}$
define an MSIA, and assume $\gamma_{0}\le\eta_{0}$. Then for any
integer $r$ and $x\in S$,
\end{lemma}
\[
\mathcal{P}^{\gamma}\left[\mathcal{A}_{r}(x)\right]\le\mathbb{P}^{\eta}\left[\mathcal{A}_{r}(x)\right].\]

\begin{proof}
We look at the ARW defined by $\eta$. Recall that regardless of $\lambda$
(the rate by which particles fall asleep in the ARW process), if two
particles or more occupy the same vertex, they will all be active
almost surely. The general consequence is as follows. Let $x$ be
a vertex that initially has $\eta_{0}(x)$ particles on it, and let
$m(x)$ be the set of distinct times at which a particle hops \emph{to}
$x$. Writing $\alpha(x)$ for the number of distinct times at which
a particle hops \emph{from} $x$, we have that with probability one,
for any $x$, $\alpha(x)\ge\left|m(x)\right|+\eta_{0}(x)-1$. For
each $x\in S$ label the active particles that will hop from $x$
by the time order of the hop, where the particle to hop first is labeled
$1$, the second to hop $2$ and so on. We now implement the $\gamma$-MSIA
as a marginal of the $\eta$-ARW in a way that each move of each explorer
is paired to a unique jump of a distinct particle. Essentially, when
an explorer is at a vertex $v$ we assign to him the first particle
that jumped from $v$ that wasn't previously assigned to him or an
explorer before him. The explorer then follows that particle. The
formal proof is below.

Let $\gamma_{i}(\cdot)$ where $i\in\mathbb{N}\cup\left\{ 0\right\} $,
denote the number of explorers at a vertex after the $i$'th explorer
completed his exploration. Next, let the discrete time variable $t$
start at $0$ and count the moves of the explorers, accumulating from
one explorer to the next. So if the first explorer made $\tau$ moves
before stopping, the first move of the second explorer is at time
$\tau$. Let $\beta_{t}(x)$ be a counter for the number of distinct
times an explorer has jumped from $x$ up to time $t$. Thus $\beta_{0}(x)=0$
for all $x$. We now describe the algorithm that chooses the path
of the $\left(i+1\right)$'th explorer starting at a vertex $x$ with
$\gamma_{i}(x)>1$ until reaching a vertex $y$ with $\gamma_{i}(y)=0$.
For time $t$ assume the current explorer is at a vertex $v$ with
$\gamma_{i}(v)>0$. Then for all $x\in S$ set $\beta_{t+1}(x)=\beta_{t}(x)+\mathbf{1}_{\left\{ x=v\right\} }$,
and move the explorer to the vertex to which the particle labeled
$\beta_{t+1}(v)$ jumped, assuming for now that such a particle exists.
Continue thus until the explorer reaches a vertex $y$ with $\gamma_{i}(y)=0$.
Set $\gamma_{i+1}(x)=\gamma_{i}(x)-1$, $\gamma_{i+1}(y)=1$ and $\gamma_{i+1}(v)=\gamma_{i}(v)$
for $v\notin\left\{ x,y\right\} $.

It is clear that each explorer move is paired to a unique jump since
we increase $\beta_{t}(x)$ at every move from $x$. It is left to
show that at the end of the process after $T$ time steps $\beta_{T}(x)\le\alpha(x)$
for all $x$ and that the MSIA is well defined. $\beta_{t}(x)$ is
incremented is two cases. In the first case, an explorer has just
hopped to $x$ from another vertex (at time $t-1$). This implies
there is a unique time in $m(x)$ corresponding to this explorer.
The second case is that the explorer began his walk at $x$, which
happens $\gamma_{0}(x)-1$ times. Since $\alpha(x)\ge\left|m(x)\right|+\eta_{0}(x)-1$
and $\eta_{0}\ge\gamma_{0}$ we are done.\end{proof}
\begin{remark}
\label{rem:Sole_ARW_assumption}The only assumption we used from the
ARW is that if a vertex is occupied by more than one particle, there
must be a time when one of the particles hops from that vertex.
\end{remark}
We introduce some more notation used in the paper. For a set $A\subset S$,
let $\eta_{0}|_{A}:S\to\mathbb{N}\cup\left\{ 0\right\} $ be equal
to $\eta_{0}$ on the set $A$ and $0$ outside of it. 

When $P$ is implicit, let $X(t)$ denote the discrete time Markov
chain according to kernel $P$.

For a set $A\subset S$, write $\tau_{A}=\left\{ \inf t\ge0:X(t)\in A\right\} $
for the random time it takes the walk to hit $A$. If $A$ contains
a single vertex $x$, write $\tau_{x}$ for $\tau_{\left\{ x\right\} }$.

Let $\left\{ A_{m}\right\} _{m\in\mathbb{N}}$ be a rising sequence
of sets containing $\mathbf{0}$. When there is no ambiguity, we write
$\tau_{m}$ for $\tau_{A_{m}^{c}}$ - the first exit time of $A_{m}$,
and $\eta_{m}$ for $\eta_{0}|_{A_{m}}$.

Write $p_{x}^{m}=P_{x}\left[\tau_{\mathbf{0}}<\tau_{m}\right]$ for
the probability a random walk on $S$ according to the kernel $P$
and starting at $x\in A_{m}$ hits $\mathbf{0}$ before exiting $A_{m}$.
Note we use the same $P$ for the kernel and the probability on $X(t)$. 

Let $\Lambda_{m}=\sum_{x\in A_{m}}p_{x}^{m}$ and let $\Omega_{m}=\sum\limits _{x\in A_{m}}\eta_{0}(x)p_{x}^{m}$.
Write $\Delta_{m}=\Omega_{m}-\Lambda_{m}$ for the difference. 
\begin{theorem}
\label{thm:Delta_deviation}Given a triplet $\eta=\left(S,\eta_{0},P\right)$,
$\eta$ is nonfixating at \textbf{$\mathbf{0}\in S$} if there exists
a sequence of sets $\left\{ A_{m}\right\} _{m\in\mathbb{N}}$ containing
$\mathbf{0}$ such that\textup{ \begin{equation}
\lim_{m\to\infty}\frac{\Delta_{m}}{\Lambda_{m}^{1/2}}=+\infty.\label{eq:deviation_condition}\end{equation}
}\end{theorem}
\begin{proof}
Since $\Lambda_{m}\ge p_{\mathbf{0}}^{m}=1$, \eqref{eq:deviation_condition}
implies that $\Delta_{m}\to\infty$. Writing $\mathcal{P}^{m}$ for
$\mathcal{P}^{\left(S,\eta_{m},P\right)}$, by Lemma \ref{lem:coupling_with_MSIA}
it is thus enough to show that for $r(m)=\lfloor\Delta_{m}/3\rfloor$,
\begin{equation}
\lim_{m\to\infty}\mathcal{P}^{m}\left[\mathcal{A}_{r(m)}\right]=1.\label{eq:many_visits_prob_to_1}\end{equation}
That is, we show that the probability $\mathbf{0}$ will be visited
at least $\lfloor\Delta_{m}/3\rfloor$ times by an explorer in the
MSIA process defined by $\left(S,\eta_{m},P\right)$ goes to one with
$m$.

The condition in \eqref{eq:deviation_condition} also implies the
stronger condition that\begin{equation}
\lim_{m\to\infty}\frac{\Delta_{m}}{\Omega_{m}^{1/2}}=\infty.\label{eq:dev_condition2}\end{equation}
To see this, assume first that $\Delta_{m}>\Lambda_{m}$ in which
case $\Delta_{m}^{2}/\Omega_{m}=\Delta_{m}^{2}/(\Delta_{m}+\Lambda_{m})>\Delta_{m}/2$.
Otherwise $\Delta_{m}^{2}/\Omega_{m}=\left(\Delta_{m}^{2}/\Lambda_{m}\right)\cdot\Lambda_{m}/(\Lambda_{m}+\Delta_{m})\ge\Delta_{m}^{2}/2\Lambda_{m}$.
By assumption, both $\Delta_{m}$ and $\Delta_{m}^{2}/\Lambda_{m}$
go to infinity with $m$, and we get \eqref{eq:dev_condition2}.

Fix $k>0$ and choose $m$ large enough so that $\Delta_{m}>3k\Omega_{m}^{1/2}$.
Using the idea from the original IDLA paper \cite{lawler1992internal},
we let the explorers positioned by $\eta_{m}$ start the walks one
at a time but assume that once an explorer reaches an unoccupied vertex
and remains there, his ghost continues to walk forever. Thus to each
explorer we associate a \emph{walk} that begins together with the
explorer but continues indefinitely. Let $W=W(m)$ be the number of
walks that visit $\mathbf{0}$ before exiting $A_{m}$, and let $L=L(m)$
be the number of walks that visit $\mathbf{0}$ before exiting $A_{m}$,
but do this as ghosts (i.e. after stopping in the original model).
$W-L$ thus counts the explorers that visit $\mathbf{0}$ before stopping
and before exiting $A_{m}$. Letting $F(m)$ be the event $\left\{ W-L<\Delta_{m}/3\right\} $,
to prove \eqref{eq:many_visits_prob_to_1} it is enough to show that
the probability of $F(m)$ goes to zero with $m$. We write $\mathcal{P}\left[\cdot\right]=\mathcal{P}(m)\left[\cdot\right]$
for the law of the MSIA with ghosts.

$W$ is a sum of independent variables with mean $\Omega_{m}$. $E\left[L\right]$
is hard to calculate, but note that each ghost that contributes to
$L$ can be tied to the unique point at which it turns from an explorer
into a ghost. Thus, by the Markov property, if we start an independent
walk from each vertex in $A_{m}$ and let $\hat{L}$ be the number
of such walks that hit $\mathbf{0}$ before exiting $A_{m}$, we have
$\mathcal{P}[L\ge a]\le\mathcal{P}[\hat{L}\ge a]$.

In particular, $E[L]\le E[\hat{L}]=\Lambda_{m}$ and we have that
$E\left[W\right]-E\left[L\right]\ge\Delta_{m}$. Thus we upperbound
$F(m)$ with the union bound \begin{equation}
\mathcal{P}\left[\Omega_{m}-W>\Delta_{m}/3\right]+\mathcal{P}\left[L-\Lambda_{m}>\Delta_{m}/3\right].\label{eq:W_L_close_to_omega_Lambda}\end{equation}
To get a bound on the probability of deviation from the mean, note
that each of $W$ and $\hat{L}$ is a sum of independent indicators
and thus their variances are bounded by their means. By this and the
Chebyshev inequality, \[
\mathcal{P}\left[\Omega_{m}-W>\Delta_{m}/3\right]\le\mathcal{P}\left[\Omega_{m}-W>k\Omega_{m}^{1/2}\right]\le\mathcal{P}\left[\Omega_{m}-W>k\sigma_{W}\right]\le k^{-2}\]
 and similarly \[
\mathcal{P}\left[L-\Lambda_{m}>\Delta_{m}/3\right]\le\mathcal{P}\left[\hat{L}-\Lambda_{m}>\Delta_{m}/3\right]\le k^{-2}.\]
Since $k$ was arbitrary we are done.
\end{proof}

\section{\label{sec:Nonfixation-for-Random}Nonfixation for Random Initial
Occupation}

The above theorem treats the case of a fixed initial occupation. From
here on, we take $\eta_{0}$ to be random, using boldcase $\mathbf{P}$
and $\mathbf{E}$ to denote the probability and expectation corresponding
to the law of the initial occupation.

For the formulation of below theorem, it is simpler to view $S$ as
an infinite connected network, i.e. a graph where each edge $e\in E(S)$
is assigned a conductance $c(e)$. For a vertex $x\in S$, let $\mathcal{E}(x)$
be the edges in $S$ that have $x$ as an endpoint. Let $\pi(x)=\sum\limits _{e\in\mathcal{E}(x)}c(e)$.
We call $\pi(x)$ the weight of $x$. Let $P$ be the transition kernel
of a simple random walk (SRW) on the network $S$. That is, unlike
SRW on a graph in which the transition probability from a vertex $x$
to a neighbor is the reciprocal of the degree of $x$, in SRW on a
network, the probability to move from $x$ to a neighbor connected
by edge $e$ is $c(e)/\pi(x)$. 

We say $S$ is $\gamma$-bounded if there is a constant $\gamma>0$
such that $\gamma<\pi(x),c(e)<\gamma^{-1}$ for each vertex $x$ and
edge $e$ of $S$. A SRW on a $\gamma$-bounded network includes diverse
examples such as SRW on a bounded degree graph and a large class of
bounded range walks on transitive graphs.
\begin{theorem}
\label{thm:uncor_mean_gt_1}Let $S$ be a $\gamma$-bounded network,
and let $P$ be the kernel of an SRW on $S$. Let $\left\{ \eta_{0}(x)\right\} _{x\in S}$
be uncorrelated r.v.'s, all with uniformly bounded variance $V<\infty$
and mean $1+\epsilon$ for some $\epsilon>0$. Then $\mathbf{P}$-almost
surely , $\eta=\left(S,\eta_{0},P\right)$ is nonfixating. \end{theorem}
\begin{proof}
First we show that without any assumptions on $S$ or $P$, if there
is a vertex $\mathbf{0}\in S$ and a sequence of sets $\left\{ A_{m}\right\} _{m\in\mathbb{N}}$
containing $\mathbf{0}$ such that $\Lambda_{m}\to\infty$, then $\eta=\left(S,\eta_{0},P\right)$
is nonfixating at $\mathbf{0}$, $\mathbf{P}$-almost surely.

We may assume (by taking a subsequence) that $\left\{ A_{m}\right\} _{m\in\mathbb{N}}$
is such that $\Lambda_{m}>m^{2}$. By assumptions the variance of
$\Omega_{m}=\sum\limits _{x\in A_{m}}\eta_{0}(x)p_{x}^{m}$ is bounded
by $V\sum\limits _{x\in A_{m}}\left(p_{x}^{m}\right)^{2}\le V\Lambda_{m}$.
Thus by Chebyshev, \[
\mathbf{P}\left[\left|\mathbf{E}\left[\Omega_{m}\right]-\Omega_{m}\right|>\frac{\epsilon}{2}\Lambda_{m}\right]<\frac{4V}{\epsilon^{2}\Lambda_{m}}<cm^{-2}.\]
By Borel Cantelli, $\mathbf{P}$-almost surely, for all large enough
$m$, $\Omega_{m}$ is not far from its mean and we have for all large
enough $m$,\begin{equation}
\Delta_{m}=\Omega_{m}-\Lambda_{m}>\frac{\epsilon}{2}\Lambda_{m}.\label{eq:delta_order_of_lambda}\end{equation}
Since $\Lambda_{m}\to\infty$ the condition in Theorem \ref{thm:Delta_deviation}
holds and we are done.

Back to our $\gamma$-bounded network $S$, by above it is enough
to show that for an arbitrary vertex $\mathbf{0}\in S$, and $A_{m}=\left\{ x\in S:d_{S}(x,\mathbf{0})<m\right\} $,
where $d_{S}(\cdot,\cdot)$ is graph distance, we have $\Lambda_{m}\to\infty$.
For $x,y\in S$ and $m\in\mathbb{N}$, define the Green's function
as \[
G_{m}(x,y)=\mathbb{E}_{x}\left[\sum_{t=0}^{\tau_{m}-1}1_{\left\{ X(t)=y\right\} }\right]\]
where $X(t)$ is our discrete time SRW as defined above. By standard
Markov chain theory, \[
p_{x}^{m}=G_{m}(x,\mathbf{0})/G_{m}(\mathbf{0},\mathbf{0}).\]
Next, for a simple random walk on a network (see e.g. chapter 2 in
\citealp{lyons1997probability}), \[
G_{m}(x,\mathbf{0})\pi(x)=G_{m}(\mathbf{0},x)\pi(\mathbf{0}).\]
By our assumption that $\pi(\cdot)$ is uniformly bounded away from
zero and infinity on $S$, we may sum over $A_{m}$ to get that for
some $c>0$, \begin{equation}
\Lambda_{m}=\sum\limits _{x\in A_{m}}p_{x}^{m}\ge cG_{m}(\mathbf{0},\mathbf{0})^{-1}E\left[\tau_{m}\right].\label{eq:lambda_as_E_to_G}\end{equation}
Using the notation of Lemma \ref{lem:E_to_G}, if we take $Z=S\backslash A_{m}$
and $x=\mathbf{0}$, we have $G_{Z}=G_{m}(\mathbf{0},\mathbf{0})$.
Note that for any $m>1$, $A_{m}$ contains all the neighbors of $\mathbf{0}$
and the conditions of the lemma are satisfied. By \eqref{eq:E_gt_G}
we have\[
G_{m}(\mathbf{0},\mathbf{0})<(E[\tau_{m}]\log G_{m}(\mathbf{0},\mathbf{0})/k(\gamma))^{\frac{1}{2}}.\]
Since $a>\log a$ for any positive $a$, again from \eqref{eq:E_gt_G}
we get $G_{m}(\mathbf{0},\mathbf{0})<E[\tau_{m}]/k(\gamma)$ and plugging
into above we have \[
G_{m}(\mathbf{0},\mathbf{0})\le\left(E[\tau_{m}]\log\frac{E[\tau_{m}]}{k}/k\right)^{1/2}.\]
Since $E[\tau_{m}]\to\infty$ we get that the right hand side of \eqref{eq:lambda_as_E_to_G}
goes to infinity and are finished. \end{proof}
\begin{remark}
Note that some uniform bound on weights and edges of $S$ is needed
to get nonfixation with the conditions of above theorem. As an example
for this, we could take $\mathbb{N}$ as the graph and set the conductances
such that the the probability for a SRW to ever hit $1$ is summable
on $\mathbb{N}$. Thus any initial occupation distributed like $\eta_{0}$
in above theorem would fixate at $1$ a.s. as can be seen by bounding
the expected number of particles to hit $1$.
\end{remark}
\begin{remark}
If $\left\{ \eta_{0}(x)\right\} _{x\in S}$ are iid with mean greater
than one, we don't need the finite variance assumption since for some
$M<\infty$, $\eta_{0}(x)\wedge M$ has mean greater than one and
all moments, thus $\eta_{0}$ is nonfixating by Theorem \ref{thm:uncor_mean_gt_1}.\end{remark}
\begin{theorem}
\label{thm:ergodic}Let $S=\mathbb{Z}^{d}$ and let $P$ be the transition
kernel of a simple random walk on $S$. Let $\left\{ \eta_{0}(x)\right\} _{x\in\mathbb{Z}^{d}}$
be distributed such that the action of the group of translations is
ergodic and such that $\mu=\mathbf{E}\left[\eta_{0}(\mathbf{0)}\right]=1+\epsilon$
for some $\epsilon>0$. Then $\mathbf{P}$-almost surely, $\eta=\left(S,\eta_{0},P\right)$
is nonfixating. \end{theorem}
\begin{proof}
Let $A_{m}=\left\{ x\in\mathbb{Z}^{d}:\|x\|_{2}<m\right\} $. As in
Theorem \ref{thm:uncor_mean_gt_1},\break $p_{x}^{m}=G_{m}(x,\mathbf{0})/G_{m}(\mathbf{0},\mathbf{0})$.
Since $\mathbb{Z}^{d}$ is regular, the Green function is symmetric
and we get $\Lambda_{m}=E\left[\tau_{m}\right]/G_{m}(\mathbf{0},\mathbf{0})$
which tends to infinity by Lemma \ref{lem:E_to_G} as shown in Theorem
\ref{thm:uncor_mean_gt_1}. Thus by Theorem \ref{thm:Delta_deviation}
it is enough to show that $\mathbf{P}$-almost surely, \begin{equation}
\Omega_{m}/\Lambda_{m}\to\mu.\label{eq:weighted_erg_lem}\end{equation}
We prove for $d>2$. The same proof works for $d=2$ using the estimate
$G_{m}(\mathbf{0},\mathbf{0})=\frac{2}{\pi}\ln n+O(n^{-1})$ (see
Proposition 1.6.7 of \citealp{lawler1996intersections}), and for $d=1$
where $G_{m}(\mathbf{0},\mathbf{0})=m$ . By the optional stopping
theorem with the martingale $\|X(t)\|^{2}-t$, $\mathbb{E}\left[\tau_{m}\right]=m^{2}+o(m^{2})$.
Second, $G_{m}(\mathbf{0},\mathbf{0})^{-1}\to\sigma$ where $\sigma(d)$
is the escape probability from $\mathbf{0}$. $\sigma$ is positive
since $d>2$. Note that in the $d=1,2$ cases, $G_{m}(\mathbf{0},\mathbf{0})^{-1}$
needs to be multiplied by $m$ and $\ln m$ respectively in order
to converge to a positive value. Continuing, we have \begin{equation}
\Lambda_{m}/m^{2}\to\sigma>0.\label{eq:exit_time_asymp}\end{equation}
Next, define the average $Q_{m}=\left|A_{m}\right|^{-1}\sum_{x\in A_{m}}\eta_{0}(x)$.
By Theorem 1.2 in \cite{lindenstrauss2001pointwise}, since $\mathbb{Z}^{d}$
is amenable and $A_{m}$ is a tempered Følner sequence, we have the
following law of large numbers \begin{equation}
Q_{m}\overset{\mathbf{P}-a.e.}{\longrightarrow}\mu.\label{eq:mean_erg}\end{equation}
For $x\in\mathbb{Z}^{d}$, write $\lceil x\rceil$ for the smallest
integer strictly greater than the Euclidean norm of $x$. Thus $\lceil\mathbf{0}\rceil=1$.
Let $c_{d}=\frac{2}{d-2}\sigma\omega_{d}^{-1}$ where $\omega_{d}$
is the volume of the unit sphere in $\mathbb{R}^{d}$. For $k\in\mathbb{N}$
let \[
q_{k}^{m}=c_{d}\left[k^{2-d}-m^{2-d}\right].\]
and for $x\in\mathbb{Z}^{d}$ let $q_{x}^{m}=q_{\lceil x\rceil}^{m}$.
By Proposition 1.5.9 in \cite{lawler1996intersections},\[
p_{x}^{m}=q_{x}^{m}+O(\lceil x\rceil^{1-d}).\]
Summing on level sets, and using that for $k\in\mathbb{N}$, $k^{1-d}-(k+1)^{1-d}<ck^{-d}$,
we get \begin{eqnarray*}
\sum_{x\in A_{m}}\eta_{0}(x)\left[p_{x}^{m}-q_{x}^{m}\right] & < & cQ_{m}\left|A_{m}\right|m^{1-d}+c\sum_{k=1}^{m-1}Q_{k}\left|A_{k}\right|k^{-d}\\
 & < & c\left[Q_{m}m+\sum_{k=1}^{m-1}Q_{k}\right].\end{eqnarray*}
Since $\Lambda_{m}$ is order of $m^{2}$ then by \eqref{eq:mean_erg}
\[
\sum_{x\in A_{m}}\eta_{0}(x)\left[p_{x}^{m}-q_{x}^{m}\right]/\Lambda_{m}\overset{a.e.}{\longrightarrow}0,\]
and to calculate the limit in \eqref{eq:weighted_erg_lem}, we can
replace $p_{x}^{m}$ in $\Omega_{m}=\sum\limits _{x\in A_{m}}\eta_{0}(x)p_{x}^{m}$
by $q_{x}^{m}$. Since $q_{m}^{m}=0$, we can write\begin{eqnarray*}
\sum_{x\in A_{m}}\eta_{0}(x)q_{x}^{m} & = & \sum_{k=1}^{m-1}Q_{k}\left|A_{k}\right|\left[q_{k}^{m}-q_{k+1}^{m}\right]\\
 & = & \sum_{k=1}^{m-1}Q_{k}\omega_{d}\left[k^{d}+o(k^{d})\right]c_{d}(d-2)\left[k^{1-d}+o(k^{1-d})\right]\\
 & = & 2\sigma\sum_{k=1}^{m-1}Q_{k}\left[k+o(k)\right].\end{eqnarray*}
Applying \eqref{eq:mean_erg} we get that \[
\sum_{x\in A_{m}}\eta_{0}(x)q_{x}^{m}/\sigma m^{2}\overset{a.e.}{\longrightarrow}\mu.\]

\end{proof}

\section{Fixation for infinite sleep rate}

In this section we prove that for the ARW process with $\lambda=\infty$,
if a deterministic initial occupation satisfies a {}``density bound''
condition around a fixed vertex, only a finite number of distinct
particles ever visit this vertex, almost surely. We then use this
to show there are nontrivial iid distributions on $\mathbb{Z}^{d}$
(including Poisson iid) that satisfy this condition, hence showing
together with Theorem \ref{thm:uncor_mean_gt_1} and the monotonicity
result in \cite{rolla2009absorbing}, that the critical density for
$\lambda=\infty$ is positive for all dimensions.

Let $\eta=\left(S,\eta_{0},P\right)$ define an ARW process. Fix $\mathbf{0}\in S$.
For a set of vertices $A\subset S$, let the weight of $A$ be $w(A)=\sum_{x\in A}\eta_{0}(x)$.
Let $W(n)$ denote the maximal weight among all connected vertex sets
of size $n$ that include $\mathbf{0}$. Let \[
\mathcal{A}(S,\eta_{0})=\sup\left\{ n\in\mathbb{N}:W(n)\ge n\right\} .\]

For the process defined by $\eta$, let $\mathcal{C}(T)$ be the random
set of vertices visited by any one of the particles by time $T\in[0,\infty]$
. Let $\mathcal{C}_{\mathbf{0}}(T)$ be the (possibly empty) connected
component of $\mathcal{C}(T)$ containing $\mathbf{0}$. 
\begin{theorem}
\label{thm:ARW_fixate}For the ARW process with infinite sleep rate
on $\eta=\left(S,\eta_{0},P\right)$, where $P$ is such that only
nearest neighbor moves are allowed,\begin{equation}
\mathbb{P}_{(\infty)}^{\eta}\left[\left|\mathcal{C}_{\mathbf{0}}(\infty)\right|\le\mathcal{A}(S,\eta_{0})\right]=1.\label{eq:bnd_size_C0}\end{equation}
In particular, since any particle that visits $\mathbf{0}$ is in
$\mathcal{C}_{\mathbf{0}}(\infty)$, and we assume $\eta_{0}$ is
locally finite, if $\mathcal{A}<\infty$ then only a finite number
of distinct particles ever visit $\mathbf{0}$ almost surely.\end{theorem}
\begin{proof}
Since $\lambda=\infty$, every vertex of $\mathcal{C}_{\mathbf{0}}(T)$
is occupied by a distinct particle at time $T$. Next, note that since
$P$ only allows nearest neighbor jumps, any particle in $\mathcal{C}_{\mathbf{0}}(T)$
at time $T$ was in $\mathcal{C}_{\mathbf{0}}(T)$ at time $0$. Hence
w.p. $1$, \[
w\left(\mathcal{C}_{\mathbf{0}}(T)\right)\ge\left|\mathcal{C}_{\mathbf{0}}(T)\right|.\]
So if $\mathcal{C}_{\mathbf{0}}(T)$ is finite then by definition
of $\mathcal{A}$, $\left|\mathcal{C}_{\mathbf{0}}(T)\right|\le\mathcal{A}$
w.p. $1$. We assume $\mathcal{A}<\infty$ as otherwise \eqref{eq:bnd_size_C0}
holds trivially. We calculate the probability for the local event
$\left\{ \left|\mathcal{C}_{\mathbf{0}}(T)\right|\le\mathcal{A}\right\} $
as a limit of its probability with initial occupation $\eta_{0}$
replace by $\eta_{m}=\eta_{0}|_{A_{m}}$ where $A_{m}\nearrow S$
are finite sets increasing to $S$. As always, we assume that $\eta_{0}$
and $P$ generate a well defined ARW process on $S$ and any local
event measurable up to a finite time can be finitely approximated.
For all $m$ and all $T<\infty$ \[
\mathbb{P}_{(\infty)}^{\left(S,\eta_{m},P\right)}\left[\left|\mathcal{C}_{\mathbf{0}}(T)\right|\le\mathcal{A}(S,\eta_{0})\right]=1.\]
Thus we get \eqref{eq:bnd_size_C0} for $\mathcal{C}_{\mathbf{0}}(T)$
for any $T<\infty$ and hence for $\mathcal{C}_{\mathbf{0}}(\infty)$
by monotonic convergence. \end{proof}
\begin{remark}
The combinatorical nature of the proof allows extension of the result
from ARW to a controllable process in which an adversary attempts
to bring infinitely many particles to $\mathbf{0}$ while observing
the rule that a particle may be moved only if it is not alone at a
vertex. Let $D(S,\eta_{0})=\limsup_{n\to\infty}W(n)/n$. In the setting
of the controllable process, $D<1$ implies finitely many visits while
$D>1$ allows for infinitely many visits, thus in a sense the condition
is sharp.
\end{remark}
Below we prove there exist nontrivial iid distribution on $\mathbb{Z}^{d}$
for which $\mathcal{A}(\mathbb{Z}^{d},\eta_{0})$ is finite almost
surely. For one dimension fixation in this setting was already known
and proven in \cite{rolla2009absorbing}.
\begin{corollary}
\label{cor:Fixation_for_d_1}Let $S=\mathbb{Z}$ and let $\left\{ \eta_{0}(x)\right\} _{x\in\mathbb{Z}}$
be an ergodic distribution on $\mathbb{N}\cup\left\{ 0\right\} $
with mean less than one. Let $P$ be a transition kernel on $S$ allowing
only nearest neighbor jumps. Then for the ARW process with $\lambda=\infty$,
$\mathbf{P}$-almost surely, $\eta=\left(S,\eta_{0},P\right)$ is
fixating.\end{corollary}
\begin{proof}
By the law of large numbers, $W(n)<n$, for all large enough $n$
almost surely. Thus $\mathbf{P}$-almost surely, $\mathcal{A}(\mathbb{Z}^{d},\eta_{0})$
is finite and we have fixation. \end{proof}
\begin{corollary}
Let $S=\mathbb{Z}^{d}$ for $d\ge2$, and let $\mathcal{D}$ be a
probability mass function on $\mathbb{N}\cup\left\{ 0\right\} $ such
that for some constant $c>0$,\[
\sum_{n=0}^{\infty}\left(1-\sum_{i=0}^{n}\mathcal{D}(i)\right)^{1/d}<c.\]
Let $P$ be a transition kernel on $S$ allowing only nearest neighbor
jumps. Set the initial occupation $\left\{ \eta_{0}(x)\right\} _{x\in\mathbb{Z}^{d}}$
to be iid r.v.'s with distribution $\mathcal{D}$. Then for the ARW
process with $\lambda=\infty$, $\mathbf{P}$-almost surely, $\eta=\left(S,\eta_{0},P\right)$
is fixating. In particular, when $\lambda=\infty$, the critical density
for the Poisson iid initial occupation is positive for all dimensions.\end{corollary}
\begin{proof}
Theorem 1 in \cite{martin2002linear} (see also \citealp{cox1993greedy})
implies that with some $c>0$, for any distribution satisfying above
condition, we almost surely have $W(n)<n$ for all large enough $n$.
Thus $\mathbf{P}$-almost surely, $\mathcal{A}(\mathbb{Z}^{d},\eta_{0})$
is finite and we have fixation.

By Proposition 8.1 in \cite{martin2002linear} for some $c'$, \[
\sum_{n=0}^{\infty}\left(1-\sum_{i=0}^{n}\mathcal{D}(i)\right)^{1/d}<c'\mathbf{E}\left[\mathcal{D}^{d+1}\right].\]
For a Poisson distribution with intensity $\gamma$, the right hand
side of above tends to $0$ as $\gamma$ approaches $0$ and we are
done. 
\end{proof}

\section{Further Remarks and Questions }
\begin{enumerate}
\item There are many variations which can be analyzed using the above framework.
For example, given a function $f:S\to\mathbb{N}$ from a graph to
the naturals, let ARW($f$), be the ARW process where there must be
more than $f(x)$ particles at a vertex $x$ to know they are all
active. Thus the usual process is ARW($\mathbf{1}$). Lemma \ref{lem:coupling_with_MSIA}
and Theorem \ref{thm:Delta_deviation} can be modified by redefining
MSIA and $\Lambda_{m}$ according to $f$, and similar theorems on
nonfixation for fixed or random initial occupations can be proved.
\item What can be shown if the random walks of particles are dependent?
(e.g. an exclusion process) A review of the proof of Theorem \ref{thm:Delta_deviation}
shows we use a property of independent indicators that could arise
in weakly dependent situations as well. Namely that the variance of
$W$ and $\hat{L}$ is the same order as their mean. A different relationship
between these quantities could imply a similar theorem with an appropriate
update of \eqref{eq:deviation_condition}. 
\item Is there an example of a graph which has nonfixation for iid initial
occupation with mean smaller than one for infinite sleep rate? For
fixed finite sleep rate, can one find the exact critical density for
some graph with iid initial occupation? What are achievable values
for critical density for Cayley graphs? For general graphs?
\item The model is not an attractive particle system and thus it is not
clear whether there exist stationary distributions for nonfixating
scenarios or what could be possible candidates. Can one find a graph
with a stationary distribution for the ARW? 
\item When fixation occurs on a graph, what can we say about the speed with
which vertex fixation spreads?
\end{enumerate}
For more questions, conjectures and numerical results, see \cite{dickman2010activated}.

\noindent {\bf Acknowledgment:} 
Thanks to Gady Kozma and Vladas Sidoravicius for telling me about
this problem and for very helpful discussions and ideas. Thanks to
Itai Benjamini for suggesting some more directions and to a referee
for a careful reading which helped weed out problems in the original
manuscript.

\section{Appendix}

Here we prove a general lemma for simple random walks on a network
relating the average exit time from a set starting at some vertex
to the average number of visits to that vertex before exit. Let $S$
be a network, let $Z\subset S$ be a set of vertices and let $\tau_{Z}$
be the first hitting time of $Z$ for $X(t)$ the discrete time simple
random walk on $S$. For $x\in S$ we set $G_{Z}(x)=E_{x}\left[\sum_{t=0}^{\tau_{Z}}\mathbf{1}_{\left\{ X(t)=x\right\} }\right]$,
the expected number of visits to $x$ of a walk starting at $x$ before
$\tau_{Z}$. Let $B_{r}(x)=\left\{ v\in S:d_{S}(v,x)<r\right\} $
and let $\partial B_{r}(x)=\left\{ v\in S:d_{S}(v,x)=r\right\} $where
$d_{S}(\cdot,\cdot)$ is graph distance in $S$.
\begin{lemma}
\label{lem:E_to_G}Assume $S$ is an infinite $\gamma$-bounded connected
network (i.e. there is a $\gamma>0$ for which $\gamma<c(e),\pi(x)<\gamma^{-1}$
for every vertex $x$ and edge $e$ in $S$). Then there is a $k=k(\gamma)>0$
such that for any $x\in S$ and $Z\subset S$ where $B_{2}(x)\cap Z=\emptyset$.
\begin{equation}
E_{x}\left[\tau_{Z}\right]>kG_{Z}^{2}/\log G_{Z}.\label{eq:E_gt_G}\end{equation}
\end{lemma}
\begin{proof}
Fix $x\in S$, set $T_{0}=0$, and define for each $i\in\mathbb{N}$
the r.v.'s \[
T_{i}=\inf\left\{ t>T_{i-1}:X(t)=x\right\} .\]
Let $i^{*}=\inf\left\{ i:T_{i}=\infty\right\} $. For $1\le i\le i^{*}$
let $\rho_{i}=T_{i}-T_{i-1}$. We show there are positive constants
$k_{1},k_{2}$ dependent only on $\gamma$ such that 

\begin{equation}
P\left[\rho_{1}\ge k_{1}r^{2}/\log r\right]\ge\frac{k_{2}}{r}.\label{eq:excursion_len_low_bnd}\end{equation}
By electrical network interpretation (see e.g. \citealp{lyons1997probability}),
the probability for a walk beginning at $x$ to hit $\partial B_{r}(x)$
before returning to $x$ is $C_{eff}(r)/\pi(x)$, where $C_{eff}(r)$
is the effective conductance from $x$ to $\partial B_{r}(x)$. Since
$S$ is infinite and connected, for any $r$ there is a connected
path of $r$ edges from $x$ to $\partial B_{r}(x)$. By the monotonicity
principle, $C_{eff}(r)$ is at least the conductance on this path,
which is $\gamma r^{-1}$. Thus the probability to hit some $y\in\partial B_{r}(x)$
before returning to $x$ is at least $\gamma^{2}r^{-1}$.

Next, let $y\in\partial B_{r}(x)$. By the Carne-Varopoulos upper
bound (see \citealp{varopoulos1985long}), \[
P\left[X(t)=x|X(0)=y\right]\le2\left(\pi(y)/\pi(x)\right)^{1/2}\exp\left(-\frac{r^{2}}{2t}\right)\]
and thus, for some $k_{1}(\gamma),k_{2}(\gamma)>0$ and all $r>1$,
the probability that a walk starting at $y\in\partial B_{r}(x)$ does
not hit $x$ in the next $\lfloor k_{1}r^{2}/\log r\rfloor$ steps,
by union bound, is greater than $k_{2}$. Together with our lower
bound on the probability that we arrive at such a $y\in\partial B_{r}(x)$,
we get \eqref{eq:excursion_len_low_bnd}.

Next, let $g=\inf\left\{ i:T_{i}>\tau_{Z}\right\} $ be the number
of visits of $X(t)$ to $x$ before hitting $Z$, including $t=0$.
$g$ is a geometric random variable with mean $G=G_{Z}$. Let $\alpha=\frac{1}{2}\ln\left(4/3\right)$
and note that since there is a constant in \eqref{eq:E_gt_G} and
$G\ge1+\gamma^{4}$, as $P[X(2)=x|X(0)=x]\ge\gamma^{4}$, we can assume
that $G>2$. Thus \begin{eqnarray*}
P\left[g\ge\alpha G\right] & \ge & \left(1-G^{-1}\right)^{\alpha G}\\
 & \ge & \left(1-G^{-1}\right)^{2\alpha\left(G-1\right)}\\
 & \ge & e^{-2\alpha}=3/4\end{eqnarray*}
We further assume $G>\frac{2}{\alpha}\vee\frac{16}{k_{2}}$ so that
$\alpha G-1>\alpha G/2$ and $G\frac{k_{2}}{16}>1$. Let $A$ be the
event that there is an $1\le i\le\left(\alpha G-1\right)\wedge i^{*}$
such that $\rho_{i}>k_{1}\left(\frac{k_{2}}{16}G\right)^{2}/\log\left(\frac{k_{2}}{16}G\right)$.
Note that $i^{*}\le\alpha G-1$ implies $A$. Thus by \eqref{eq:excursion_len_low_bnd}
and the independence of consecutive excursions from $x$, \[
P\left[A^{c}\right]\le\left(1-\frac{16}{G}\right)^{\alpha G/2}\le e^{-8\alpha}\]
which is smaller than $1/4$.

Thus \[
P\left[\left\{ g\ge\alpha G\right\} ,A\right]\ge\frac{1}{2}.\]
This implies the lemma since $\tau_{Z}>\sum_{i=1}^{g-1}\rho_{i}$,
and for $k=k_{1}k_{2}^{2}/256$, $\sum_{i=1}^{g-1}\rho_{i}>kG^{2}/\log G$.
\end{proof}

\bibliographystyle{alea2}
\bibliography{07-07}

\end{document}